\documentclass[12pt]{amsart}
\usepackage{amsmath}
\usepackage{amsthm}
\usepackage{amsfonts}
\usepackage{amssymb}
\newtheorem{Theorem}{Theorem}[section]
\newtheorem{Proposition}[Theorem]{Proposition}
\newtheorem{Lemma}[Theorem]{Lemma}
\newtheorem{Corollary}[Theorem]{Corollary}
\theoremstyle{definition}
\newtheorem{Definition}{Definition}[section]
\theoremstyle{remark}

\numberwithin{equation}{section}

\newcommand{\Z}{{\mathbb Z}}
\newcommand{\R}{{\mathbb R}}
\newcommand{\C}{{\mathbb C}}
\newcommand{\N}{{\mathbb N}}

\begin{document}

\title[Reflectionless Herglotz functions]{Reflectionless Herglotz functions
and generalized Lyapunov exponents}

\author{Alexei Poltoratski}

\address{Mathematics Department\\
Texas A\&M University\\
College Station, TX 77843}

\email{alexei@math.tamu.edu}

\urladdr{www.math.tamu.edu/$\sim$alexei.poltoratski/}

\author{Christian Remling}

\address{Mathematics Department\\
University of Oklahoma\\
Norman, OK 73019}

\email{cremling@math.ou.edu}

\urladdr{www.math.ou.edu/$\sim$cremling}

\date{May 28, 2008}



\begin{abstract}
We study several related aspects of reflectionless Jacobi matrices.
Our first set of results deals with the singular part of reflectionless measures.
We then introduce and discuss Lyapunov exponents, density of states measures, and other related
quantities in a \textit{general }setting. This is related to the previous
material because the density of states measures are reflectionless on certain sets.
\end{abstract}
\maketitle
\section{Introduction}
We study several aspects of reflectionless Jacobi matrices and Herglotz
functions in this paper. This is part of a larger program; the (perhaps
too ambitious) goal is to reach a systematic understanding of the absolutely
continuous spectrum of Jacobi operators $J$ on $\ell_2(\Z_+)$,
\[
(Ju)(n) = a(n)u(n+1)+a(n-1)u(n-1)+b(n)u(n) .
\]
We will always assume that the coefficients $a$, $b$ satisfy bounds of the form
\[
(C+1)^{-1} \le a(n) \le C+1, \quad |b(n)|\le C ,
\]
for some $C>0$. Note that if $J$ has some absolutely continuous spectrum,
then, by the decoupling argument of Dombrowski and Simon-Spencer \cite{Dom,SimSp},
it actually suffices to assume that $a(n)$ is bounded above;
the other two inequalities follow automatically.

Let us recall some definitions.
A \textit{Herglotz function }is a holomorphic mapping of
$\C^+=\{ z\in\C : \textrm{ Im }z >0\}$ to itself. We denote the set of Herglotz
functions by $\mathcal H$. If $F\in\mathcal H$, then $F(t)\equiv \lim_{y\to 0+} F(t+iy)$
exists for (Lebesgue) almost every $t\in\R$. We call $F$
\textit{reflectionless }(on $E\subset\R$) if
\begin{equation}
\label{refll}
\textrm{\rm Re }F(t) = 0 \quad \textrm{\rm for almost every }t\in E .
\end{equation}
We will also use the notation
\[
\mathcal N(E) = \{ F\in\mathcal H: F \textrm{ \rm reflectionless on }E \} .
\]
Herglotz functions have unique representations of the form

\begin{equation}
\label{Fmu}
F(z) = F_{\mu}(z) =
a + bz + \int_{-\infty}^{\infty} \left( \frac{1}{t-z} - \frac{t}{t^2+1} \right) \, d\mu(t) ,
\end{equation}
with $a\in\R$, $b\ge 0$, and a (positive) Borel measure $\mu$ on $\R$,
$\int_{\R} \frac{d\mu(t)}{t^2+1}<\infty$.
We will call such a measure $\mu$ \textit{reflectionless }(on $E$)
if $F_{\mu}\in \mathcal N(E)$ for some choice of $a\in\R$, $b\ge 0$; for easier
reference, it will also be convenient to introduce the notation
\[
\mathcal R(E) = \{ \mu : \mu \textrm{ reflectionless on }E \} .
\]
We emphasize again that in particular $\int_{\R}\frac{d\mu(t)}{t^2+1}<\infty$
for all $\mu\in\mathcal R(E)$. Also, if $\mu\in\mathcal R(E)$, then
$F_{\mu}$ will refer to the unique Herglotz function $F_{\mu}\in\mathcal N(E)$
that is associated with $\mu$ as in \eqref{Fmu}.

There are several reasons for being interested in the class $\mathcal N(E)$;
here, our main motivation is provided by the following fact: Call a (whole line) Jacobi matrix $J$
\textit{reflectionless }(on $E$) if $g_n\in\mathcal N(E)$ for all $n\in\Z$, where
\[
g_n(z) = \langle \delta_n, (J-z)^{-1} \delta_n \rangle
\]
is the $n$th diagonal element of the resolvent of $J$ (also known as the Green function).
Then \cite[Theorem 1.4]{Remac} says that all $\omega$ limit points of a Jacobi matrix $J$ with some
absolutely continuous spectrum are reflectionless on $E=\Sigma_{ac}$; here, $\Sigma_{ac}$
denotes an essential support of the absolutely continuous part of the spectral measure $\rho$
of $J$. This is defined up to sets of Lebesgue measure zero; we can obtain
a representative as $\Sigma_{ac}=\{t: d\rho/dt > 0 \}$.
Please see \cite{Remac} for the details.

In particular, we see that the class $\mathcal N(E)$
(for a bounded, essentially closed set $E\subset\R$) is an important object in the study of
the asymptotic behavior of the coefficients of Jacobi matrices $J$ with $\sigma_{ess}=
\Sigma_{ac}=E$. Statements addressing these issues
are sometimes referred to as \textit{Denisov-Rakhmanov Theorems}
(after \cite{Dencont,Den,Rakh}); at present, only the cases where $E$ is a
finite union of closed intervals \cite{DKS,Remac} or, more generally,
a so-called homogenous set \cite{Remac,SodYud} are well understood.

If $\mu\in\mathcal R(E)$, then
$\chi_E\, dt \ll d\mu_{ac}$. Indeed, this follows immediately from \eqref{refll}
because $d\mu_{ac}(t)=(1/\pi)\textrm{Im }F(t)\, dt$ and the boundary value of a
Herglotz function can not be zero on a set of positive measure. However, it is not
so clear in general if $\mu$ can also have a singular part on $E$. We have the
following criterion. We say that a (positive)
measure $\nu$ is supported by a (measurable)
set $S$ if $\nu(S^c)=0$; unless explicitly stated otherwise,
supports are \textit{not }assumed to be closed in this paper.
\begin{Theorem}
\label{T1.1}
Let $\mu\in\mathcal R(E)$. Then:\\
(a) $\mu_s$, the singular part of $\mu$, is supported by
\[
\left\{ x\in\R : \lim_{h\to 0+} \frac{|E\cap (x-h,x+h)|}{2h} = 0 \right\} .
\]
(b) Let $\theta\in L_{\infty}(E)$ be an arbitrary bounded measurable
function. Then $\mu_s$ is also supported by
\[
\{ x\in\R : (\widetilde{H}_E\theta)(x) \textrm{ \rm exists} \} .
\]
\end{Theorem}
Here, we define $\widetilde{H}_Ef$ as
\[
(\widetilde{H}_E f)(x) = \lim_{y\to 0+} \int_E \left( \frac{t-x}{(t-x)^2 + y^2} -
\frac{t}{t^2+1} \right) \, f(t)\,dt ,
\]
if the limit exists.
This is closely related to the \textit{Hilbert transform}
\[
(Hf)(x) = \lim_{y\to 0+} \int_{|t-x|>y} \frac{f(t)}{t-x}\, dt .
\]
For instance, if $f\in L_1(\R)$, then both $\widetilde{H}_Ef$ and $H\chi_Ef$ exist
(Lebesgue) almost everywhere and define the same function (almost everywhere),
up to an additive constant.
Here, we are interested
in the \textit{singular }part of $\mu$, so sets of Lebesgue measure zero do matter,
and we distinguish between the two transforms. However, for a bounded, integrable
function, the difference between the integrals defining $\widetilde{H}_E$ and $H$ stays bounded,
so we also obtain the following variant of Theorem \ref{T1.1}(b):\\
\textit{(c) Let $\theta\in L_{\infty}(E)$. Then $\mu_s$ is supported by}
\[
\left\{ x\in\R : \sup_{0<y\le 1} \left| \int_{y<|t-x|\le 1} \frac{\theta(t)}{t-x}\,
\chi_E(t)\, dt \right| < \infty \right\} .
\]

It is tempting to compare Theorem \ref{T1.1}(a)
with the results of Sodin-Yuditskii \cite{SodYud}
on reflectionless Jacobi matrices with homogeneous spectrum. Sodin and Yuditskii
prove that these operators are almost periodic and have purely absolutely continuous
spectra. We can obtain this latter conclusion for the much larger class of
\textit{weakly }homogeneous spectra (to be defined in Section 2) directly
from Theorem \ref{T1.1}(a); see Definition \ref{D2.1} and Corollary \ref{C2.1} below.
However, this will not give any additional information on
the structure of the associated Jacobi matrices. Perhaps Theorem \ref{T1.1}
can actually be used as the starting point for such a refined analysis, but at present,
this is only a hope for the future.

It is easier to keep track of the pure point part of a reflectionless measure, and
in fact Theorem \ref{T1.1}(c) gives a sharp criterion in this case.
\begin{Theorem}
\label{T1.2}
Let $E\subset\R$ be a Borel set, and fix $x\in\R$. Then the following are equivalent:\\
(a) $\mu(\{x\})=0$ for all $\mu\in\mathcal R(E)$.\\
(b)
\[
\int_{x-1}^{x+1} \frac{\chi_E(t)}{|t-x|}\, dt = \infty
\]
\end{Theorem}
Our final result on the singular part of reflectionless measures is of a conditional nature.
It says that if $\mu$ is also non-zero outside $E$, then this will only make it more difficult
to produce a singular part on $E$. To be able to formulate this concisely, we introduce
\[
\mathcal R_0(E) = \left\{ \mu\in \mathcal R(E): \mu(E^c)=0 \right\} .
\]
\begin{Theorem}
\label{T1.3}
Let $E\subset\R$ be a closed set. Suppose that
$\mu_s(E)=0$ for $\mu\in\mathcal R_0(E)$. Then $\nu_s(E)=0$ for {\rm all}
$\nu\in\mathcal R(E)$.
\end{Theorem}

As our next topic, we would like to address the following question: Given a set $E$,
how can we produce examples of measures that are reflectionless on $E$? Two quick
answers are immediately available: As already mentioned above, \cite[Theorem 1.4]{Remac}
says that if we start out with a Jacobi matrix with $\Sigma_{ac}\supset E$ and then
take $\omega$ limit points, then we can be sure that these will be reflectionless on $E$.
A different answer to our question was obtained in \cite[Theorem 5.4]{MelPVol}
(see also \cite{NazVolYud}):
the potential theoretic equilibrium measure is reflectionless on its support.

These two results are not totally unrelated. More precisely, the equilibrium measure
frequently arises as the density of states measure of Jacobi matrices with some
absolutely continuous spectrum. Observations of this type are not new; see \cite{Simpot}
for a recent survey on the use of potential theoretic notions. We will develop this
and related material quite systematically in the last three sections of this paper.

This discussion will begin in Section 3.
We will introduce Lyapunov exponents, density of states measures and
other related objects in a \textit{general }setting; these quantities, of course,
are in common use, but only for \textit{ergodic }operators.
We will do this in the obvious way by taking limits on subsequences.

The following general result on the Lyapunov exponent $\gamma$, which will be
established in Section 4, is the main reason why these quantities are of interest
to us here.
We will make use of the function $g(z)=\int\frac{dk(t)}{t-z}$, where $dk$
is a density of states measure. Please see Section 3 for the precise definitions.
Also recall that since $g\in\mathcal H$, the limit $g(x)\equiv\lim_{y\to 0+} g(x+iy)$
exists for almost every $x\in\R$.

We will need the following definition.
\begin{Definition}
\label{D4.1}
Let $f:\R\to\R$ be a (Lebesgue) measurable function. We say that $f$ is
\textit{approximately differentiable }at $x\in\R$ if there exists $d\in\R$ so that
\[
\lim_{h\to 0+}\frac{1}{2h} \left| \left\{
y\in(-h,h): \left| \frac{f(x+y)-f(x)}{y} - d \right| \ge \epsilon
\right\} \right| = 0
\]
for all $\epsilon>0$. In this case, we call $d$ the \textit{approximative derivative }of
$f$ at $x$, and we write $(D_{\textrm{ap}}f)(x)=d$.
\end{Definition}
Please see \cite{Bruck,Saks,Thom} for much more on this and related topics.
\begin{Theorem}
\label{T4.2}
For almost all $x\in\R$, we have that
\begin{equation}
\label{Dapgamma}
(D_{\textrm{\rm ap}}\gamma)(x)=-\textrm{\rm Re}\; g(x) .
\end{equation}
In particular, $\gamma$ is approximately differentiable almost everywhere.
\end{Theorem}
This is a development of \cite[Theorem 5.4]{MelPVol}. See also \cite{NazVolYud} for
subsequent work inspired by the same result.
Theorem \ref{T4.2} may be viewed as a result on interchanging limits because,
as we will discuss in Section 3, $\textrm{Re }g(x+iy)=-\partial_x \gamma(x+iy)$ for $y>0$,
so, for almost every $x\in\R$, $\textrm{Re }g(x)=-\lim_{y\to 0+} \partial_x\gamma(x+iy)$.
This raises the question of whether it is possible to perform these operations in the opposite
order; in other words, can we first take the boundary value of $\gamma$ to obtain
$\gamma(x)$ and then take the derivative? Theorem \ref{T4.2} provides an affirmative answer
if the derivative is taken in the approximate sense.

However, for us here, Theorem \ref{T4.2} is significant mainly because it identifies
sets on which $g$ is reflectionless; in particular, this set will contain
the points of constancy of $\gamma$. More precisely, we obtain the following:
\begin{Corollary}
\label{C4.1}
Let
\[
K = \left\{ c\in\R : \left| \gamma^{-1}(\{ c \} ) \right| > 0 \right\} ;
\quad\quad C=\gamma^{-1}(K) .
\]
Then $g\in\mathcal N(C)$.
\end{Corollary}
In particular, this will imply that
\[
g\in\mathcal N(\Sigma_{ac});
\]
equivalently, density of states measures are always reflectionless on $\Sigma_{ac}$. See
Theorem \ref{T4.1} below for this conclusion.
\begin{proof}
$K$ is countable and thus $C$ is an at most countable union
of sets $C_j$ of the form $C_j=\gamma^{-1}(\{ c_j\} )$. Almost every point of $C_j$
is a point of density, and at such points, clearly $D_{\textrm{ap}}\gamma = 0$.
Theorem \ref{T4.2} now gives the Corollary.
\end{proof}
The method that we will use to prove Theorem \ref{T4.2} also gives
the following companion result on the regularity of $\gamma$.
In analogy to Definition \ref{D4.1}, introduce
\[
A(x,\epsilon;f) = \{ y\in\R: |f(y)-f(x)|\ge \epsilon\} ,
\]
and call a (measurable) function $f$
\textit{approximately continuous }at $x\in\R$ if
\[
\frac{1}{2h} \left| A(x,\epsilon;f) \cap (x-h,x+h) \right| \to 0
\quad\quad (h\to 0+)
\]
for all $\epsilon >0$. See again \cite{Bruck,Saks,Thom} for background information.
We will show that $\gamma$ satisfies the following related, but considerably stronger condition.
\begin{Theorem}
\label{T4.3}
For every $x\in\R$,
\[
\textrm{\rm cap}\left[ A(x,\epsilon;\gamma)\cap (x-h,x+h)\right] = o(h^N)
\quad\quad (h\to 0+)
\]
for arbitrary $N\ge 1$, $\epsilon>0$.
\end{Theorem}
Here, $\textrm{cap}(S)$ denotes the logarithmic capacity of a set $S$;
please consult \cite{Ran} for the definition and basic properties.

In fact, Theorem \ref{T4.3} is not really new: $\gamma$ is subharmonic and
thus continuous in the fine topology, so $A(x,\epsilon;\gamma)$ cannot have a
fine accumulation point at $x$, and now Wiener's criterion for thinness
\cite[Theorem 5.4.1]{Ran} can be used to deduce Theorem \ref{T4.3}. Our treatment
is direct and does not depend on any machinery, so perhaps it is of some interest also.

If $S\subset\R$ is a Borel set, then $|S|\le 4\, \textrm{cap}(S)$
(see \cite[Theorem 5.3.2(c)]{Ran}), so Theorem \ref{T4.3} implies that
\[
\left| A(x,\epsilon;\gamma)\cap (x-h,x+h)\right| = o(h^N)
\quad\quad (h\to 0+)
\]
for arbitrary $N\ge 1$, $\epsilon>0$. In particular, $\gamma(x)$ is approximately continuous
at all points $x\in\R$.

To round off our discussion, we will also gather a few applications of potential
theoretic notions (capacities, equilibrium measures) in the final section of this
paper. While these applications
are straightforward and also very similar (or identical) to previous work
(see \cite{Simpot,StT} for more
systematic expositions of this),
they do seem to illuminate our earlier discussions.
In particular, we hope that the material from Section 5 will
reinforce a point we are trying to make in Sections 3 and 4,
namely, that Lyapunov exponents and density of states measures can be very useful
tools for \textit{general }Jacobi matrices.

Among other things, we will point out that the equilibrium measure frequently
occurs as the (unique) density of states measure, and this measure is reflectionless
on its support. We also collect some general inequalities and inclusions involving
the support of the density of states and the sets $\Sigma_{ac}$, $\{t\in\R : \gamma(t)=0 \}$.
Please see Section 5 for the details.

\textit{Acknowledgment: }We thank David Damanik for illuminating discussions on the
theory of ergodic operators.
\section{The singular part of reflectionless measures}
Please recall the notation $F_{\mu}$ introduced in \eqref{Fmu}. Also, if
$f\ge 0$ is a Borel function, then, as expected, $f\mu$ will
denote the measure $(f\mu)(A)=\int_A f\, d\mu$. The following result
from \cite{Pol} will be our main tool in this section.
\begin{Theorem}[\cite{Pol}]
\label{TPol}
\[
\lim_{y\to 0+} \frac{F_{f\mu}(x+iy)}{F_{\mu}(x+iy)} = f(x)
\]
for $\mu_s$-almost every $x\in\R$.
\end{Theorem}
A clarifying comment is in order: Given $\nu$, the function $F_{\nu}$ is
of course not completely determined yet (we don't know $a$, $b$). This, however,
is not an issue here; the statement from the Theorem holds for \textit{all }such
functions. This follows because $|F_{\mu}(x+iy)|\to\infty$ as $y\to 0+$ for
$\mu_s$-almost every $x\in\R$. In Theorem \ref{TPol}, we of course implicitly assume
that $1/(t^2+1)$ is integrable for all measures involved here.
See also \cite{CMT,JL} for further discussion of this theorem.

We will also use the following consequence of Theorem \ref{TPol}.
\begin{Proposition}
\label{P2.1}
Suppose that $\rho=\rho_s$ and $\sigma\perp\rho$. Then
\[
\lim_{y\to 0+} \frac{F_{\sigma}(x+iy)}{F_{\rho}(x+iy)} = 0
\]
for $\rho$-almost every $x\in\R$.
\end{Proposition}
\begin{proof}
Pick a Borel set $T\subset\R$ with $\rho(T^c)=\sigma(T)=0$, and abbreviate
$\rho+\sigma=\mu$. Then
\[
\frac{F_{\sigma}}{F_{\rho}} = \frac{F_{\chi_{T^c}\mu}}{F_{\chi_T\mu}}
= \frac{F_{\chi_{T^c}\mu}}{F_{\mu}}\, \frac{F_{\mu}}{F_{\chi_T\mu}} ,
\]
and $F_{\chi_{T^c}\mu}/F_{\mu}\to \chi_{T^c}$ $\mu_s$-almost everywhere by Theorem \ref{TPol}.
In particular, this ratio goes to zero $\rho$-almost everywhere. Similarly,
$F_{\chi_T\mu}/F_{\mu}\to 1$ $\rho$-almost everywhere, so the Proposition follows.
\end{proof}
\begin{proof}[Proof of Theorem \ref{T1.1}]
Let $\mu\in\mathcal R(E)$. Write $F_{\mu}$ for the associated Herglotz
function $F_{\mu}\in\mathcal N(E)$, as in \eqref{Fmu}, and
let $\xi$ be the Krein function of $F_{\mu}$, that is,
\[
\xi(x) = \frac{1}{\pi} \lim_{y\to 0+} \textrm{Im}\, \ln F_{\mu}(x+iy) ,
\]
where we take the logarithm with $0<\textrm{Im}\, \ln w<\pi$ for $w\in\C^+$.
Since $\ln F_{\mu}$ is a Herglotz function, the limit defining $\xi$ exists
almost everywhere and $0\le \xi(x)\le 1$.

If, conversely, a measurable function $\zeta$ with values in $[0,1]$ is given, then
$\zeta$ is the Krein function of some Herglotz function $G$. We can in fact recover
$\ln G$, up to an additive real constant, from $\zeta$, using the Herglotz representation
of $\ln G$. Here, we make use of the fact that since $\ln G$ has bounded imaginary part,
the associated measure is purely absolutely continuous.

The condition that $F_{\mu}\in\mathcal N(E)$ means that $\xi=1/2$ (almost everywhere) on $E$.
Given an arbitrary function $\theta\in L_{\infty}(\R)$,
with $-1\le\theta\le 1$ and $\theta=0$ on $E^c$,
we can therefore introduce two new Krein functions $\xi_{\pm}$, as follows:
\[
\xi_{\pm}(x) = \xi(x) \pm \frac{1}{2}\theta(x)
\]
As just explained, this also defines two new Herglotz functions $F_{\pm}$,
up to multiplicative constants.
We fix these constants by demanding that $|F_{\pm}(i)|=|F_{\mu}(i)|$. Call the measures
associated with these functions $\mu_+$ and $\mu_-$, respectively.
Since $\xi = (\xi_+ +\xi_-)/2$, we then have that
\begin{equation}
\label{Fpm}
F_{\mu} = \sqrt{ F_{\mu_+} F_{\mu_-}} .
\end{equation}

Our first aim is to show that
\begin{equation}
\label{mupmac}
\mu_s\ll \mu_{\pm} .
\end{equation}
Suppose this were wrong and write
\[
\mu_s = g\mu_{+,s}+ \nu ,
\]
with $\nu\perp \mu_{+,s}$, $\nu\not= 0$. We can then find
a Borel set $T$ so that $\nu(T)>0$, $\mu_{+,s}(T)=\nu(T^c)=|T|=0$. Theorem \ref{TPol} now
shows that
\[
\frac{F_{\nu}(x+iy)}{F_{\mu_s}(x+iy)} = \frac{F_{\chi_T\mu_s}(x+iy)}{F_{\mu_s}(x+iy)} \to 1
\]
for $\mu_s$-almost every $x\in T$, and, similarly,
\[
\frac{F_{\nu}(x+iy)}{F_{\mu_++\nu}(x+iy)} =
\frac{F_{\chi_T(\mu_++\nu)}(x+iy)}{F_{\mu_++\nu}(x+iy)} \to 1
\]
for $(\mu_{+,s}+\nu)$-almost every $x\in T$ and thus also for $\mu_s$-almost every $x\in T$.
Put differently, this means that
\[
\frac{F_{\mu_+}(x+iy)}{F_{\nu}(x+iy)} \to 0
\]
for $\mu_s$-almost every $x\in T$. So on a set of positive $\mu_s$-measure,
\begin{equation}
\label{2.1}
\frac{F_{\mu_+}(x+iy)}{F_{\mu_s}(x+iy)} =
\frac{F_{\mu_+}}{F_{\nu}} \frac{F_{\nu}}{F_{\mu_s}} \to 0 .
\end{equation}
We also have that for $\mu_s$-almost every $x\in\R$,
\begin{equation}
\label{2.2}
\sup_{0<y\le 1} \left| \frac{F_{\mu_-}(x+iy)}{F_{\mu_s}(x+iy)} \right| < \infty .
\end{equation}
This follows quickly from Proposition \ref{P2.1} with $\rho=\mu_s$ if we write
$\mu_-=h\mu_s+\sigma$, with $\sigma\perp\mu_s$. Indeed, $F_{\mu_-}/F_{\mu_s}\to h$
at $\mu_s$-almost every point by the Proposition and Theorem \ref{TPol}, and
$h<\infty$ $\mu_s$-almost everywhere.

Finally, Theorem \ref{TPol} also implies that
\[
\lim_{y\to 0+} \frac{F_{\mu_s}(x+iy)}{F_{\mu}(x+iy)} = 1
\]
for $\mu_s$-almost every $x\in\R$, and if this is combined with \eqref{2.1}, \eqref{2.2},
we obtain that
\[
\frac{\sqrt{F_{\mu_+}F_{\mu_-}}}{F_{\mu}} \to 0
\]
on a set of positive $\mu_s$-measure, but by \eqref{Fpm}, this ratio is
identically equal to one, so we reach a contradiction if \eqref{mupmac} fails.

Thus we can write
\[
\mu_s = f_{\pm}\mu_{\pm} = f_{\pm}\mu_{\pm , s} ,
\]
with $f_{\pm}\ge 0$ and in fact $0<f_{\pm}<\infty$ at
$\mu_s$-almost all points. By Theorem \ref{TPol},
\[
\lim_{y\to 0+} \frac{F_{\mu_s}(x+iy)}{F_{\mu_{\pm}}(x+iy)} \to f_{\pm}(x+iy)
\]
for $\mu_{\pm,s}$-almost every $x\in\R$ and thus also $\mu_s$-almost everywhere.
It follows that for $\mu_s$-almost every $x$,
\begin{equation}
\label{2.3}
\lim_{y\to 0+} \frac{F_{\mu_+}(x+iy)}{F_{\mu_-}(x+iy)} \quad
\textrm{exists and is positive;}
\end{equation}
in fact, this limit is equal to $f_-(x)/f_+(x)$ $\mu_s$-almost everywhere.

By definition of $\xi_{\pm}$, we have that $\xi_+-\xi_-=\theta$, so, if we introduce
\[
L(z) = \int_E \left( \frac{1}{t-z} - \frac{t}{t^2+1}\right)\theta(t) \, dt ,
\]
then $F_{\mu_+}/F_{\mu_-} = e^L$. Since $-1\le\theta\le 1$, we have that
$\textrm{Im }L(z)\in (-\pi,\pi)$ on $\C^+$, and thus \eqref{2.3} implies that for
$\mu_s$-almost every $x\in\R$,
\[
L(x)\equiv \lim_{y\to 0+} L(x+iy) \:\;\; \textrm{exists,} \quad
\textrm{Im }L(x)=0 .
\]
In particular, if we take $\theta=\chi_E$, then
\[
\textrm{Im }L(x+iy) = \int_E \frac{y}{(t-x)^2+y^2}\, dt \ge \frac{1}{2y}
\left| E \cap (x-y,x+y) \right| ,
\]
so part (a) of the Theorem follows. Part (b) is now also immediate, from the
fact that $\textrm{Re }L(x+iy)$ approaches a finite limit as $y\to 0+$.
\end{proof}
The set from part (a) of Theorem \ref{T1.1} contains $\overline{E}^c$,
so this result really addresses the question of whether $\mu$ can have a singular
part on $\overline{E}$.
In particular, it says that no such singular part can be present for
closed sets of the following type.
\begin{Definition}
\label{D2.1}
Call a Borel set $E\subset\R$
\textit{weakly homogeneous }if
\[
\limsup_{h\to 0+} \frac{1}{2h} \left| E\cap (x-h,x+h) \right| >0
\]
for all $x\in E$.
\end{Definition}
This condition is much weaker than the following, which
is used to define \textit{homogeneous }sets:
\[
\inf_{x\in E} \inf_{0<h\le 1} \frac{1}{2h} \left| E\cap (x-h,x+h) \right| >0 .
\]
From the work of Sodin-Yuditskii \cite{SodYud} it was previously known that if $E$ is
a compact (strongly) homogeneous set
and $\mu\in\mathcal R_0(E)$, then $\mu_s=0$. By using
Theorem \ref{T1.1}, we can go considerably beyond this:
\begin{Corollary}
\label{C2.1}
Suppose that $E$ is a weakly homogeneous set. If $\mu\in\mathcal R(E)$, then $\mu_s(E)=0$.
\end{Corollary}
This is more general in two respects: $E$ is only assumed to be weakly homogeneous
(rather than homogeneous), and we can treat measures from $\mathcal R(E)$,
not just from $\mathcal R_0(E)$. This latter improvement, of course, can also be
obtained from the general principle that we formulated as Theorem \ref{T1.3}.

We now move on to proving
Theorem \ref{T1.2}. This will follow quickly from the following known characterization
of the point part of $\mu$ in terms of the Krein function $\xi$ of $F_{\mu}$.
See, for example, \cite[pg.\ 201]{MP}. We include a proof for the reader's convenience.
\begin{Lemma}
\label{L2.2}
$\mu(\{ x\})>0$ if and only if
\begin{equation}
\label{2.6}
\int_{x-1}^{x+1} \frac{|\xi(t)-\chi_{(x,\infty)}(t)|}{|t-x|}\, dt < \infty .
\end{equation}
\end{Lemma}
\begin{proof}
First of all, we can recover the point part as
\[
\mu(\{x \}) = -i \lim_{y\to 0+} y F_{\mu}(x+iy) ;
\]
this is well known and follows quickly from the dominated convergence theorem. So $\mu(\{ x\})>0$
if and only if
\begin{equation}
\label{2.5}
\limsup_{y\to 0+} \left( \textrm{\rm Re}\,\ln F_{\mu}(x+iy) + \ln y \right) > -\infty .
\end{equation}
To slightly simplify the notation, we will now assume that $x=0$. In terms of the Krein function $\xi$,
the expression from \eqref{2.5} equals
\[
\int_{-1}^1 \frac{t}{t^2+y^2}\, \xi(t)\, dt - \int_y^1 \frac{dt}{t} + O(1) \quad\quad (y\to 0+) .
\]
By monotone convergence,
\[
\int_{-1}^0 \frac{t}{t^2+y^2}\, \xi(t)\, dt \to \int_{-1}^0 \frac{\xi(t)}{t}\, dt
\]
(and, of course, this limit could equal $-\infty$). Also,
\begin{align*}
\int_0^1 \frac{t}{t^2+y^2}\, \xi(t)\, dt - \int_y^1 \frac{dt}{t} & =
\int_y^1 \frac{t(\xi(t)-1)}{t^2+y^2}\, dt - \int_y^1 \frac{y^2}{t(t^2+y^2)}\, dt \\
& \quad\quad
+ \int_0^y \frac{t\xi(t)}{t^2+y^2}\, dt \\
& = \int_y^1 \frac{t(\xi(t)-1)}{t^2+y^2}\, dt + O(1) ,
\end{align*}
and, by monotone convergence again,
\[
\int_y^1 \frac{t(\xi(t)-1)}{t^2+y^2}\, dt \to \int_0^1 \frac{\xi(t)-1}{t}\, dt \ge -\infty .
\]
These calculations have shown that \eqref{2.5}, for $x=0$, holds if and only if
\[
\int_{-1}^0 \frac{\xi(t)}{|t|}\, dt + \int_0^1 \frac{1-\xi(t)}{t}\, dt < \infty ,
\]
as asserted by the Lemma.
\end{proof}
\begin{proof}[Proof of Theorem \ref{T1.2}]
Suppose that condition (b) from Theorem \ref{T1.2} fails. Put
\[
\xi(t) = \frac{1}{2} \chi_E(t) + \chi_{E^c\cap (x,\infty)}(t) .
\]
Let $F\in\mathcal H$ be the corresponding Herglotz function. Since $\xi=1/2$ on $E$,
we have that $F\in\mathcal N(E)$, but it is also clear that \eqref{2.6} holds,
so the corresponding measure has a point mass at $x$.

The converse is an immediate consequence of Theorem \ref{T1.1}(c), with
$\theta(t)=\textrm{sgn}(t-x)$. Furthermore, we can also obtain this statement
conveniently from Lemma \ref{L2.2}, as follows:
If $F_{\mu}\in\mathcal N(E)$, then $\xi=1/2$ almost everywhere on $E$, so
the integrand from \eqref{2.6} equals $1/(2|t-x|)$ on $E\cap (x-1,x+1)$ and thus
\eqref{2.6} can not hold if we have condition (b) from Theorem \ref{T1.2}.
\end{proof}
\begin{proof}[Proof of Theorem \ref{T1.3}]
Let $\nu\in\mathcal R(E)$. We claim that if $\mu\in\mathcal R_0(E)$,
then we must have that
\begin{equation}
\label{2.18}
\lim_{y\to 0+} \frac{F_{\mu}(x+iy)}{F_{\nu}(x+iy)} = 0
\end{equation}
for $\nu_s$-almost every $x\in\R$. Indeed, $\mu_s=0$, $\mu(E^c)=0$ by assumption, and,
as discussed above, the condition that
$\nu\in\mathcal R(E)$ forces the absolutely continuous part of $\nu$
to be equivalent to $\chi_E\, dt$ on $E$. So $\mu\ll\nu$, $\mu_s=0$,
and thus \eqref{2.18} follows
immediately from Theorem \ref{TPol}.

Starting from $\nu$, we will now construct a measure $\mu\in\mathcal R_0(E)$ for which
\eqref{2.18} cannot hold at any point $x\in E$. This will prove that $\nu_s(E)=0$, as claimed.

We will again work with the Krein functions; the following simple monotonicity property
is at the heart of the matter.
\begin{Lemma}
\label{L2.1}
For $\xi\in L_{\infty}(a,b)$, $0\le\xi\le 1$, and $x\notin [a,b]$, define
\[
I_x(\xi)= \int_a^b \frac{\xi(t)\, dt}{t-x} .
\]
Let $c=\int_a^b \xi(t)\, dt$. Then
\[
I_x(\xi) \le I_x\left( \chi_{(a,a+c)} \right) \quad \textrm{ for all } x\notin [a,b] .
\]
\end{Lemma}
\begin{proof}
It suffices to prove this for step functions $\xi$ because these are dense in $L_1$.
So assume that $\xi = \sum_{j=1}^N s_j \chi_{I_j}$, with disjoint intervals $I_j$.
If $(c,c+h)$ is such an interval of constancy of $\xi$ and $\xi=s$ on $(c,c+h)$, with $0<s<1$, then,
as an elementary argument shows, $I_x(\xi)$ will go up if we redefine $\xi$ on $(c,c+h)$ as
$\chi_{(c,c+sh)}$.
Use this procedure on all intervals of constancy. Since $I_x(\xi)$ clearly also increases if
we pass to the non-increasing rearrangement of $\xi$, we obtain the Lemma.
\end{proof}
Let $\xi$ be the Krein function of $F_{\nu}$, and, motivated by Lemma \ref{L2.1},
define $\xi_0$ as follows: $\xi_0=1/2$ on $E$, and if $(a,b)$ is one of the bounded components
of the open set $E^c$, set $\xi_0=\chi_{(a,a+c)}$ on $(a,b)$, where
$c=\int_a^b \xi\, dt$, as in the Lemma.
If $E^c$ has unbounded components, put $\xi_0=1$ (say) on these.
Notice that $\xi_0$ is the Krein function of a Herglotz function
$F_{\mu}$ whose associated measure satisfies $\mu\in\mathcal R_0(E)$. Indeed,
$\mu$ is reflectionless on $E$ because this property is equivalent to $\xi_0=1/2$ on $E$,
and $\mu(E^c)=0$ because $F_{\mu}(x)\equiv\lim_{y\to 0+} F_{\mu}(x+iy)$ exists and is real
at all points of $E^c$, except possibly at the jumps of $\xi_0$. However, these can't be
discrete points of $\mu$ either because in order for this to happen, $\xi_0$ would have
to jump from $0$ to $1$, not the other way around, by Lemma \ref{L2.2}.

Now fix $x\in E$ and look at $\ln |F_{\mu}/F_{\nu}|$. As $y\to 0+$,
\begin{equation}
\label{2.19}
\ln |F_{\mu}(x+iy)| - \ln |F_{\nu}(x+iy)| =
\int_{y<|t-x|\le 1} \frac{\xi_0(t)-\xi(t)}{t-x}\, dt + O(1) .
\end{equation}
Since $\xi=\xi_0=1/2$ on $E$, this set doesn't contribute to the integral. Moreover,
by Lemma \ref{L2.1} and construction of $\xi_0$, those components of $E^c$ that are
contained in the region of integration make non-negative contributions. This more or less
finishes the proof except that there might also be up to four truncated components
of $E^c$ contributing to the integral. Suppose for example that $(a,b)$ is such a component and
$x\le a<x+y<b\le x+1$. Suppose also, for simplicity, that $x=0$. We claim that then
\[
\int_y^b \frac{\xi_0(t)-\xi(t)}{t}\, dt \ge -1 .
\]
This follows because Lemma \ref{L2.1} says that this integral will only become smaller if we
replace the actual $\xi$ on $(y,b)$ by $\chi_{(y,y+h)}$, where again $h$ is chosen so that
the integral of $\xi$ over $(y,b)$ is left unchanged. A similar process was used to construct
$\xi_0$, so after this replacement, $\xi_0$ and $\xi$ are both characteristic functions of
an interval, and the interval of $\xi_0$ is not smaller than the one corresponding to $\xi$,
so the difference $\xi_0-\xi$ is zero, except perhaps on an interval of at most the size of
the truncated piece $(a,y)$, and this is obviously $\le y$.

Similar discussions of course apply to the other cases, so \eqref{2.19} is bounded below
as $y\to 0+$ and \eqref{2.18} cannot hold.
\end{proof}
\section{General Lyapunov exponents and density of states measures}
In this section, we define basic objects such as the density of states, Lyapunov exponents etc.\ in a
general setting. Usually, these quantities are considered only for Jacobi matrices that come
from an ergodic dynamical system, but pretty much the same setup also works in a general situation,
provided the limits are taken on suitable subsequences.

Let us make this precise. We are given a half line Jacobi matrix $J$ on
$\ell_2(\Z_+)$, where $\Z_+=\{1,2,\ldots \}$. It will be convenient to define $a(0)=1$.
For $z\in\C^+$, let $f_{\pm}(n,z)$ be the solutions of
\begin{equation}
\label{se}
a(n)f(n+1)+a(n-1)f(n-1)+b(n)f(n)=zf(n)
\end{equation}
that satisfy the initial conditions:
\begin{align*}
f_-(0,z) & = 0  & f_+(0,z) =1\quad \\
f_-(1,z) & = 1 & f_+(1,z) =-m_+(z)
\end{align*}
Here, $m_+(z)=\langle \delta_1, (J-z)^{-1} \delta_1 \rangle$ denotes the Titchmarsh-Weyl
$m$ function of $J$. Notice that $f_+\in\ell_2(\Z_+)$; put differently, $f_{\pm}$ are
in the domain of $J$ near $+\infty$ and $1$, respectively.
Of course, $f_-$ can be defined in this way for all $z\in\C$, and we will in fact
use this function for $z=t\in\R$ later on. Also, recall that if $z\in\C^+$,
then $f_{\pm}(n,z)\not= 0$ for all $n\ge 1$.

We write $f_{\pm}$ in polar coordinates:
\[
f_{\pm}(n,z) = R_{\pm}(n,z)e^{\mp i\varphi_{\pm}(n,z)}
\]
Here, we demand that $R_{\pm}>0$ and
\begin{equation}
\label{phi}
0<\varphi_{\pm}(n+1,z)-\varphi_{\pm}(n,z)<\pi .
\end{equation}
Moreover, the initial values are
\[
R_-(1,z)=1, \: \varphi_-(1,z)=0, \quad R_+(0,z)=1, \: \varphi_+(0,z) = 0 .
\]
These conditions can be satisfied because the functions
\begin{equation}
\label{mpm}
m_{\pm}(n,z) := \mp \frac{f_{\pm}(n+1,z)}{a(n)f_{\pm}(n,z)} =
\mp \frac{R_{\pm}(n+1,z)}{a(n)R_{\pm}(n,z)} e^{\mp i(\varphi_{\pm}(n+1,z)-\varphi_{\pm}(n,z))}
\end{equation}
are Herglotz functions (see \cite{Teschl}).

We are now ready to give the basic definitions. Define Herglotz functions $w_{\pm}^{(N)}$, $g_N$
and (probability) measures $d\nu_N$, $dk_N$, as follows:
\[
w_{\pm}^{(N)}(z) = \frac{1}{N} \sum_{n=1}^N \ln \left[ a(n)m_{\pm}(n,z) \right]
\]
More precisely, we again use the logarithm with $\textrm{Im}(\ln\zeta)\in (0,\pi)$ for
$\zeta\in\C^+$ here. It follows from \eqref{phi}, \eqref{mpm} that
\begin{align}
\label{w+}
w_+^{(N)}(z) & = \frac{1}{N} \ln R_+(N+1,z) + i \left( \pi - \frac{\varphi_+(N+1,z)}{N} \right) \\
\nonumber
& \quad\quad -\frac{1}{N}\ln R_+(1,z) + \frac{i}{N} \varphi_+(1,z) , \\
\label{w-}
w_-^{(N)}(z) & = \frac{1}{N} \ln R_-(N+1,z) + i\, \frac{\varphi_-(N+1,z)}{N} .
\end{align}
Next, let
\[
g_N(z) = \frac{1}{N} \sum_{n=1}^N \langle \delta_n, (J-z)^{-1} \delta_n \rangle =
\int_{-\infty}^{\infty} \frac{d\nu_N(t)}{t-z} ,
\]
where
\[
d\nu_N(t) = \frac{1}{N} \sum_{n=1}^N d\| E(t)\delta_n\|^2 .
\]
Here, $E$ denotes the spectral resolution of $J$. Finally, let $\lambda_n^{(N)}$
($n=1,\ldots, N$) be the eigenvalues of $J$ on $\ell_2\left(
\{ 1,\ldots, N \} \right)$ with boundary conditions
$u(0)=u(N+1)=0$, and put
\[
dk_N(t) = \frac{1}{N} \sum_{n=1}^N \delta_{\lambda_n^{(N)}} .
\]
The following pair of theorems describes basic properties of these quantities; of course,
Theorem \ref{T3.2} below is an analog of similar results for ergodic operators,
and the proof also proceeds along these well trodden paths. Analogs of other
familiar basic results (such as Thouless formula, support of the density of states etc.)
will be discussed later in this section.
\begin{Theorem}
\label{T3.1}
Fix a sequence $N_j\to\infty$, and consider the following statements:\\
{\rm (WP) }$w_+^{(N_j)}(z) \to w_+(z)$ uniformly on compact subsets of $\C^+$, for some
$w_+\in\mathcal H$.\\
{\rm (WM) }$w_-^{(N_j)}(z) \to w_-(z)$ uniformly on compact subsets of $\C^+$, for some
$w_-\in\mathcal H$.\\
{\rm (G) }$g_{N_j}(z)\to g(z)$ uniformly on compact subsets of $\C^+$, for some $g\in\mathcal H$.\\
{\rm (N) }$d\nu_{N_j}\to d\nu$ in weak $*$ sense, for some probability (Borel) measure $\nu$ on $\R$.\\
{\rm (K) }$dk_{N_j}\to d\widetilde{k}$ in weak $*$ sense, for some probability (Borel) measure
$d\widetilde{k}$ on $\R$.\\
{\rm (A) }$\left( a(1) \cdots a(N_j) \right)^{1/N_j}\to A$, for some $A>0$.\\

Then {\rm (WP)}$\iff${\rm (WM) }and {\rm (G)}$\iff${\rm (N)}$\iff${\rm (K). }Moreover,
{\rm (WP) }or {\rm (WM) }implies {\rm (G), (N), (K)}
and also {\rm (A). }Conversely, if {\rm (A) }holds, then each of {\rm (G), (N), (K)}
implies {\rm (WP), (WM).}
\end{Theorem}
Note that we will always be able to achieve convergence on suitable subsequences;
for example, we can use the Banach-Alaoglu Theorem in conditions {\rm (N), (K) }if we also
observe that the measures $d\nu_N$, $dk_N$ have supports contained in a fixed compact set.

It is natural to ask if perhaps the first five conditions are all equivalent to each other;
unfortunately, we have not been able to clarify this.
\begin{Theorem}
\label{T3.2}
Let $N_j$ be a sequence so that the conditions from Theorem \ref{T3.1} hold. Introduce
$\gamma, k: \C^+\to\R$ by writing $w_+$ as
\[
w_+(z) = -\gamma(z) + i\pi k(z) .
\]
Then $k(t)\equiv\lim_{y\to 0+} k(t+iy)$ exists for every $t\in\R$, $k(t)$ is an increasing
function, $-iw_+,w_+'\in\mathcal H$, and the following identities hold:
\begin{gather*}
w_+(z)+w_-(z) \equiv i\pi, \quad g(z)=w_+'(z), \quad d\widetilde{k}=d\nu=dk\\
g(z) = \int_{-\infty}^{\infty} \frac{dk(t)}{t-z}\\
w_+(z) = \ln A - \int_{-\infty}^{\infty} \ln (t-z)\, dk(t)
\end{gather*}
\end{Theorem}
\begin{proof} The following proof will establish both Theorem \ref{T3.1} and
Theorem \ref{T3.2}. First of all, notice that if a sequence $N_j$ is given, we can
always pass to a subsequence $N'_j$ so that all the limits from Theorem \ref{T3.1}
exist on that subsequence. Indeed, as already pointed out above, this follows from
the Banach-Alaoglu Theorem in parts (N), (K), and
in parts (WP), (WM), (G), we can use a normal families argument.
Note that in this latter case, the limiting function could in principle be identically
equal to a constant $c\in\R\cup\{ \infty \}$ (rather than lie in $\mathcal H$). However,
this does not actually happen here; we will rule out this possibility in a moment.

Now fix a sequence $N_j\to\infty$ for which all limits from Theorem \ref{T3.1} exist.
As anticipated in the statement of Theorem \ref{T3.2}, define $\gamma$, $k$ by writing
\[
w_+(z) = -\gamma(z) + i\pi k(z) \quad\quad (z\in\C^+) .
\]
Since $R_+\in\ell_2$, \eqref{w+} shows that $\gamma(z)\ge 0$. Moreover, $0\le k(z)\le 1$
and in fact both inequalities are strict unless $w_+$ is a constant. The \textit{Wronskian}
of two functions $u,v$ on $\Z_+$ is defined as
\[
W(u,v)=a(n)\left( u(n)v(n+1)-u(n+1)v(n) \right) .
\]
If $u,v$ both solve the same equation \eqref{se}, then $W(u,v)$ is independent of $n$.
From \eqref{mpm}, we obtain that
\begin{multline*}
W(f_+,f_-)= a(n)^2 R_+(n,z)R_-(n,z) e^{i(\varphi_-(n,z)-\varphi_+(n,z))} \\
\times (m_+(n,z)+m_-(n,z)) .
\end{multline*}
On a suitable subsequence $N'_j$, the coefficients $a(N'_j)$ as well as the
Herglotz functions $m_{\pm}(N'_j,\cdot)$ will also
converge, and these latter limits have to be genuine Herglotz functions (not real constants)
because the associated measures are finite measures whose supports are contained
in a fixed compact set. It follows that
\[
R_+(N'_j,z)R_-(N'_j,z)\to \alpha(z)>0
\]
as $j\to\infty$, and thus, by \eqref{w+}, \eqref{w-},
$\textrm{Re }w_+(z) = -\textrm{Re }w_-(z)$. In other words, $w_++w_-$ is
a Herglotz function whose real part is identically equal to zero. Hence
$w_++w_-\equiv iB\pi$ for some $B\ge 0$, or, equivalently,
\begin{equation}
\label{3.1}
w_-(z) = \gamma(z)+i\pi (B-k(z)).
\end{equation}
We will now use oscillation theory to prove that $B=1$ here.

Write $k(t)\equiv \lim_{y\to 0+} k(t+iy)$ and
$\varphi_-(N,t)\equiv \lim_{y\to 0+} \varphi_-(N,t+iy)$; by general facts about Herglotz functions,
these limits exist for almost every $t\in\R$. Moreover, by combining \eqref{3.1} with
\eqref{w-}, we see that
\begin{equation}
\label{3.12}
(B-k(t))\, dt = \lim_{j\to\infty} \frac{\varphi_-(N_j+1,t)}{\pi N_j}\, dt ;
\end{equation}
the limit is in weak $*$ sense, and we are also using the fact that, since
the Herglotz functions we are currently discussing have bounded imaginary parts,
their associated measures are purely absolutely continuous. Similarly, \eqref{w+}
implies that
\begin{equation}
\label{3.14}
(1-k(t))\, dt = \lim_{j\to\infty} \frac{\varphi_+(N_j+1,t)}{\pi N_j}\, dt .
\end{equation}

Suppose now that $f_-(n,t)\not=0$ for $n=1,2,\ldots, N+1$. Notice that for fixed $N$,
this will fail only at finitely many $t\in\R$. So $R_-(n,t)>0$, and the continuity
of $f_-(n,\cdot)$ together with the normalization $\varphi_-(n+1,z)-\varphi_-(n,z)
\in (0,\pi)$ (for $z\in\C^+$) now imply the following:
\[
\varphi_-(n,t) = \lim_{y\to 0+} \varphi_-(n,t+iy)\:\; \textrm{exists, }
\quad \varphi_-(1,t)=0,
\]
and
\[
\varphi_-(n+1,t)-\varphi_-(n,t) = \begin{cases} 0 &\textrm{ if }
f_-(n+1,t)f_-(n,t) > 0 \\ \pi & \textrm{ if }f_-(n+1,t)f_-(n,t) < 0 \end{cases} .
\]
In particular, $(1/\pi)\varphi_-(N+1,t)$ equals the number of sign changes
of $f_-(\cdot, t)$ on $\{ 1,2, \ldots, N+1 \}$.

By oscillation theory \cite[Chapter 4]{Teschl}, $\varphi_-(N+1,\cdot)$ is a
decreasing function; the jumps occur at the zeros of $f_-(N+1,t)$, but these
are precisely the eigenvalues $\lambda_j^{(N)}$ of the problem on $\{ 1,2, \ldots, N\}$. Therefore,
\begin{equation}
\label{3.2}
\frac{1}{\pi N}\, d(-\varphi_-(N+1,t)) = dk_N(t) .
\end{equation}

Similar reasoning may be applied to $f_+$. If $m_+(t)\equiv\lim_{y\to 0+}m(t+iy)$
exists and $f_+(n,t)\not= 0$ for $n=1,2,\ldots, N+1$, then
$\varphi_+(n,t)\equiv \lim_{y\to 0+} \varphi_+(n,t+iy)$ exists, too, and
$\varphi_+(N+1,t)$ again essentially counts sign changes. More precisely,
$(1/\pi)\varphi_+(N+1,t)$ differs by at most $2$ from the number of sign
changes of $\textrm{Re }f_+(\cdot, t)$ on $\{ 1,2,\ldots, N\}$.

By oscillation theory again, any two non-trivial solutions to the same equation
have essentially the same number of sign changes; more precisely, the difference
is at most $1$ in absolute value. Thus
\[
\frac{\varphi_-(N+1,t)}{\pi N}\, dt - \frac{\varphi_+(N+1,t)}{\pi N}\, dt \to 0
\]
as $N\to\infty$, in the weak $*$ topology. By taking the limits on the sequence $N_j$
and recalling \eqref{3.12}, \eqref{3.14}, we now see that indeed $B=1$ in \eqref{3.1}.
We have established the first identity from Theorem \ref{T3.2}.

Let $\psi\in C_0^{\infty}(\R)$.
Then, by \eqref{3.2},
\begin{align*}
\frac{1}{\pi N_j} \int_{-\infty}^{\infty} \varphi_-(N_j+1,t) \psi'(t)\, dt &
= \frac{1}{\pi N_j} \int_{-\infty}^{\infty} \psi(t)\, d(-\varphi_-(N_j+1,t))\\
& \to \int_{-\infty}^{\infty} \psi(t)\, d\widetilde{k}(t)
= - \int_{-\infty}^{\infty} \psi'(t) \widetilde{k}(t)\, dt .
\end{align*}
Here, we define the measure $d\widetilde{k}$ as the limit from part (K)
of Theorem \ref{T3.1}; we then also obtain a corresponding increasing (and,
let's say: right-continuous) function $\widetilde{k}(t)=\int_{(-\infty,t]}d\widetilde{k}(s)$.

On the other hand, we see from \eqref{3.12} that this integral also converges to
$\int_{-\infty}^{\infty} \psi'(t) (1-k(t))\, dt$, so we deduce that
$\widetilde{k}(t)=k(t)+c$, for some constant $c\in\R$. In fact, it would be more cautious to say that
we obtain this relation off a set of Lebesgue measure zero
(at this point, the only thing we can say for sure is that $k(t)\equiv\lim_{y\to 0+} k(t+iy)$
has been defined almost everywhere, and $k$, $\widetilde{k}$ might have
discontinuities). We will see later that these precautions
are actually unnecessary: $c=0$ and $k(t)$ exists everywhere and is continuous.
In any event, we now know that $dk=d\widetilde{k}$.

Next, we analyze the Herglotz representation of $w_+$.
Recall again that $\textrm{Im }w_+$ is bounded, so the associated measure is
purely absolutely continuous and the Herglotz representation reads
\[
w_+(z) = C_0 + \int_{-\infty}^{\infty} \left( \frac{1}{t-z} - \frac{t}{t^2+1} \right)
k(t)\, dt .
\]
An integration by parts shows that
\begin{align*}
w_+(z) & = C_0 + \int_{-\infty}^{\infty} \frac{\partial}{\partial t}
\left[ \ln (t-z) - \frac{1}{2}\ln(t^2 + 1) \right] k(t)\, dt \\
& = C_0 + \lim_{R\to\infty} k(t)\ln\frac{t-z}{\sqrt{t^2+1}} \Bigr|_{t=-R}^{t=R} \\
& \quad\quad
-\int_{-\infty}^{\infty} \left[ \ln (t-z) - \frac{1}{2}\ln(t^2 + 1) \right]\, dk(t)\\
& = C - \int_{-\infty}^{\infty} \ln (t-z) \, dk(t) .
\end{align*}
Some comments on this calculation are in order: As we saw above, $dk$ is a compactly
supported finite
measure on $\R$. This also implies that the function $k(t)$ approaches finite limits
as $t\to\pm\infty$. Finally, as usual, $\textrm{Im}(\ln\zeta)\in (0,\pi)$
for $\zeta\in\C^+$, and if $\zeta\in\C^-$, we interpret $\ln\zeta = -\ln\zeta^{-1}$.

We in particular obtain the following pair of formulae:
\begin{align*}
w'_+(z) & = \int_{-\infty}^{\infty} \frac{dk(t)}{t-z} ,\\
\gamma(z) & = -C + \int_{-\infty}^{\infty}\ln |t-z|\, dk(t)
\end{align*}
This latter identity also lets us identify the constant $C$: Since
$dk_{N_j}\to dk$ in the weak $*$ topology and these (probability) measures have
supports contained in a fixed compact set, we see that for $z\in\C^+$,
\begin{align*}
\gamma(z) & = -C + \lim_{j\to\infty} \int_{-\infty}^{\infty}\ln |t-z|\, dk_{N_j}(t) \\
& = -C + \lim_{j\to\infty} \frac{1}{N_j} \ln \prod_{n=1}^{N_j}
\left| z-\lambda_n^{(N_j)} \right| \\
& = -C + \lim_{j\to\infty} \frac{1}{N_j} \ln \left| a(1)\cdots a(N_j) f_-(N_j+1,z)\right| .
\end{align*}
To pass to the second line, we just used the definition of $dk_{N_j}$, and
the last equality follows because $f_-(N_j+1,z)$ is a polynomial of degree $N_j$,
with leading term $z^{N_j}/(a(1)\cdots a(N_j))$ and zeros precisely at the $\lambda_n^{(N_j)}$. Now
\[
\frac{1}{N_j} \ln |f_-(N_j+1,z)| \to \textrm{Re }w_-(z)=\gamma(z) ,
\]
so it follows that $\lim (1/N_j)\ln (a(1)\cdots a(N_j))$ exists and equals $C$. Thus
\begin{align}
\label{thou}
\gamma(z) & = - \ln A + \int_{-\infty}^{\infty}\ln |t-z|\, dk(t) ,\\
\nonumber
w_+(z) & = \ln A - \int_{-\infty}^{\infty}\ln(t-z)\, dk(t) ,
\end{align}
where we have defined
\[
A = \lim_{j\to\infty} \left( a(1)\cdots a(N_j) \right)^{1/N_j} .
\]

Identity \eqref{thou} (the \textit{Thouless formula; }see also Corollary \ref{C3.1} below)
has a number of consequences, which we now develop. First of all,
\begin{equation}
\label{defgammat}
\gamma(t) \equiv \lim_{y\to 0+} \gamma(t+iy)
\end{equation}
exists for \textit{all }$t\in\R$. Indeed, since $dk$ is a finite measure of compact support,
\[
\lim_{y\to 0+} \int_{|s-t|>1/2} \ln |s-t-iy|\, dk(s)
\]
exists by dominated convergence. Moreover, monotone convergence shows that the
integrals over $|s-t|\le 1/2$ also converge, and
\[
\lim_{y\to 0+} \int_{|s-t|\le 1/2} \ln |s-t-iy|\, dk(s) = \int_{|s-t|\le 1/2}\ln |s-t|\, dk(s)
\ge -\infty .
\]
Since $\gamma(t+iy)\ge 0$, the limit cannot be equal to $-\infty$, and
\[
\int_{\R} \ln |s-t|\, dk(s) > -\infty .
\]
Our claim follows, and we have also shown that \eqref{thou} in fact holds for $z\in\C^+\cup\R$,
if $\gamma(t)$ for $t\in\R$ is defined by \eqref{defgammat}.
Also, $k(t)$ must be continuous. (It is well known that we actually obtain the
somewhat stronger conclusion that $k$ is log H\"older continuous; see again Corollary \ref{C3.1}
below.) The Poisson integral representation for the harmonic extension $k(z)$ to $\C^+$ now
shows that $k$ is actually continuous on $\C^+\cup\R$; in particular, $k(t)=\lim_{y\to 0+}
k(t+iy)$ exists everywhere, as claimed.

Our next goal is to show that $d\nu=dk$; recall that
we have already proved that $dk=d\widetilde{k}$.
Since the supports of $dk_N$, $d\nu_N$ are contained in a fixed
compact set and since every continuous function on a compact subset of $\R$
can be uniformly approximated by polynomials, it suffices to show that
\begin{equation}
\label{3.3}
\lim_{j\to\infty} \int_{\R} t^n\, dk_{N_j}(t) = \lim_{j\to\infty} \int_{\R} t^n\, d\nu_{N_j}(t)
\end{equation}
for all $n\ge 0$. Now
\[
\int_{\R} t^n\, dk_N(t) = \frac{1}{N} \sum_{j=1}^N \left( \lambda_j^{(N)} \right)^n
= \frac{1}{N}\,\textrm{tr}\: J_N^n = \frac{1}{N}\sum_{j=1}^N \langle \delta_j, J_N^n \delta_j \rangle ;
\]
here $J_N$ denotes the restriction of $J$ to $\ell_2(\{ 1, \ldots, N \})$. In other words,
if $P_N$ denotes the projection onto this subspace, then $J_N=P_NJP_N$. On the other hand,
\[
\int_{\R} t^n\, d\nu_N(t)= \frac{1}{N}\sum_{j=1}^N \langle \delta_j, J^n \delta_j \rangle ,
\]
and since $J^n\delta_j=J_N^n\delta_j$ if $n+j\le N$, \eqref{3.3} indeed follows.

It follows that
\[
g(z) = \int_{\R} \frac{dk(t)}{t-z} .
\]
This also shows that $g=w'_+$, and we have now established all identities from
Theorem \ref{T3.2}. In particular, this implies that
each of the quantities $g$, $dk$, $d\nu$
determines the other two from this list,
and also $w_{\pm}$, up to a constant, and this constant, in turn, is determined by $A$.
Also, $w_+$ or $w_-$ clearly determines everything else.

As pointed out above, existence of each of the
limits from Theorem \ref{T3.1} can always be achieved by passing to
a suitable subsequence (of a given sequence), so the relations between these
conditions that were spelled out in Theorem \ref{T3.1} follow now.
\end{proof}
The \textit{Thouless formula} \eqref{thou} will play a particularly important role
in our subsequent discussion, so we state this again, for emphasis, and add a well known consequence.
\begin{Corollary}
\label{C3.1}
For $z\in\C^+\cup\R$, we have that
\[
\gamma(z) = -\ln A + \int_{\R} \ln |t-z|\, dk(t) .
\]
The function $k(t)$ is log H\"older continuous:
\[
k(t+h)-k(t) \le \frac{C}{-\ln h} \quad\quad (0<h\le 1/2)
\]
\end{Corollary}
For the proof of the log H\"older continuity from the Thouless formula,
see \cite{CSim}.

The Thouless formula displays the \textit{Lyapunov exponent }$\gamma$ as
the logarithmic potential of the \textit{density of states }measure $dk$,
so we may expect potential theoretic notions to be useful here.
This will be a recurring theme in the sequel. For now, we need these
tools to give a precise bound on the support of $dk$. A very accessible source
for general information about potential theory is \cite{Ran}.

In \cite{Simpot}, Simon proposes the following refinement
of the familiar decomposition of the spectrum
into essential and discrete spectrum: Denote the (logarithmic) capacity
of a set $E$ by $\textrm{cap}(E)$, and define, for Borel sets $S\subset\R$,
\begin{equation}
\label{Scap}
S_{\textrm{cap}} = \left\{ x\in\R : \textrm{\rm cap}(S\cap (x-h,x+h))>0\textrm{ for all }
h> 0 \right\} .
\end{equation}
It is clear from the definition and basic properties of capacities
(as discussed in Chapter 5 of \cite{Ran}) that this set is closed
and $\textrm{cap}(S\setminus S_{\textrm{cap}})=0$.
We will use this definition mainly for $S=\sigma$, the spectrum of the Jacobi matrix $J$.
Note that then $\sigma_{\textrm{cap}}\subset\sigma_{ess}$, and the inclusion
can be strict.
\begin{Proposition}
\label{P3.1}
If $E\subset\R$ is a Borel set with $\textrm{\rm cap}(E)=0$, then $k(E)=0$. Moreover,
$dk$ is supported by $\sigma_{\textrm{\rm cap}}$.
\end{Proposition}
\begin{proof}
The assumption that $\textrm{cap}(E)=0$
means that
\[
\int\!\!\int\ln |s-t|\, d\mu(s)\, d\mu(t) =-\infty
\]
for all positive Borel measures $\mu$ that are supported by a compact subset of $E$.
If we had $k(E)>0$, then also $k(C)>0$ for some compact $C\subset E$ by regularity,
but it follows from the Thouless formula that
\[
\int_C\int_C \ln |s-t|\, dk(s)\, dk(t) > -\infty .
\]
Thus $k(E)=0$. To prove the second claim, it now suffices to show that
$dk$ is supported by $\sigma$, because we only split off a capacity zero set when
passing to $\sigma_{\textrm{cap}}$. This, however, is clear from oscillation theory:
If $[s,t]\cap\sigma = \emptyset$, then $\varphi_-(N,t)-\varphi_-(N,s)=O(1)$ as $N\to\infty$,
so, by \eqref{3.2}, $k_N([s,t])\to 0$ and thus $k((s,t))=0$.
\end{proof}
In the theory of ergodic operators, one has that almost surely,
$\sigma=\sigma_{\textrm{cap}}$ and the topological support of
$dk$ (which is defined as the smallest closed support)
is \textit{exactly }this set. Of course, this in no longer true
in the general setting of this section, where we just take limits on
subsequences. For example, if there are sufficiently long intervals with
$a(n)=1$, $b(n)=0$, then on suitable subsequences, the operator will approximately
look like the free Jacobi matrix ($a=1$, $b=0$ identically). In particular,
$\textrm{supp }dk = [-2,2]$ for these limits, but of course $\sigma$ can be
much larger, depending on what happens on the complement of these intervals.
\section{Further properties of generalized Lyapunov exponents}
We continue our discussion of the quantities introduced in the preceding section.
Throughout this section and the next, we fix a sequence $N_j\to\infty$ for which the limits
from Theorem \ref{T3.1} exist. Note, however, that this sequence is otherwise
arbitrary, so our results will apply to \textit{all }such limits.

We have defined $\gamma(t)$ for $t\in\R$ as a limit of Lyapunov exponents
$\gamma(t+iy)$. It is natural to ask if we can also obtain $\gamma(t)$ directly
from the solutions to the Jacobi difference equation \eqref{se} for $z=t$.
\begin{Proposition}
\label{P4.1}
Let $N_j\to\infty$ be a sequence for which the conditions from Theorem \ref{T3.1} hold. Then
\[
\gamma(t) = \limsup_{j\to\infty} \frac{1}{N_j} \ln R_-(N_j+1,t)
\]
for quasi every $t\in\R$.
\end{Proposition}
A property is said to hold \textit{quasi everywhere }if it holds off a Borel set of
capacity zero.
\begin{proof}
This is an immediate consequence of the upper envelope theorem \cite[Theorem A.7]{Simpot} because,
as discussed in the proof of Theorem \ref{T3.2}, we have the identity
\[
\frac{1}{N_j} \ln R_-(N_j+1,t) = -\ln (a(1)\cdots a(N_j))^{1/N_j} +
\int_{-\infty}^{\infty} \ln |s-t|\, dk_{N_j}(s) ,
\]
and $(a(1)\cdots a(N_j))^{1/N_j}\to A$, $dk_{N_j}\to dk$ in weak $*$ sense as $j\to\infty$.
\end{proof}
Let $d\rho(t)=d\|E(t)\delta_1\|^2$ be the usual spectral measure of $J$. Then
\begin{equation}
\label{schnol}
R_-(n,t) \le C_t n \quad\quad (n\ge 1)
\end{equation}
for $\rho$-almost every $t\in\R$. This type of statement is well known and sometimes referred
to as \textit{Schnol's Theorem. }Here, \eqref{schnol} follows quickly from the identity
$\int_{\R} R_-^2(n,t)\, d\rho(t) = 1$. See \cite{LSim} for a detailed discussion of
these topics.

We need one more piece of notation: put
\[
Z=\{ t\in\R :\gamma(t) = 0 \} .
\]
\begin{Theorem}
\label{T4.1}
There exists a Borel set $N\subset\R$, with $\textrm{\rm cap}(N)=0$, so that
$\rho$ is supported by $Z\cup N$. In particular, $w_{\pm}, g\in\mathcal N(\Sigma_{ac})$.
\end{Theorem}
The statement on the support of $\rho$ is a deterministic version
of \cite[Theorem 1.16]{Simpot}.
\begin{proof}
By \eqref{schnol} and Proposition \ref{P4.1}, $\gamma(t)=0$ for $\rho$-almost every
$t\in\R$, except perhaps on the capacity
zero set where the Proposition doesn't apply. This gives the first claim.
Clearly $w_{\pm}\in\mathcal N(Z)$,
but also $g\in\mathcal N(Z)$ by Corollary \ref{C4.1} (we haven't proved this yet, but
the proof of Theorem \ref{T4.2} is
the very next item on the agenda). Since the capacity zero set $N$
cannot support absolutely continuous measures, we have that $\Sigma_{ac}\subset Z$.
\end{proof}

Our next goal is to prove Theorems \ref{T4.2} and \ref{T4.3}. For a more streamlined
presentation, we isolate the following simple (but key) calculation.
\begin{Lemma}
\label{L4.1}
If $0<|t|\le 2y$, then
\[
\int_0^y \left| \ln \left|1-\frac{h}{t} \right| \right|\, dh \le 12y\ln\left( 1 + \frac{y}{|t|}
\right) .
\]
\end{Lemma}
\begin{proof}
We will treat explicitly only the case $0<t\le 2y$ here; the case where $t<0$ is similar, but
easier. In the former case,
\begin{equation}
\label{4.12}
\int_0^y \left| \ln \left| 1 - \frac{h}{t} \right|\, \right|\, dh =
t \int_{1-y/t}^1 \left| \ln |s| \right|\, ds .
\end{equation}
If $y/t\le 2$, then
\[
\eqref{4.12} \le 2t \int_0^1 \left| \ln s \right| \, ds = 2t \le 4y .
\]
Similarly, if $y/t>2$, then
\[
\eqref{4.12} = 2t + t\int_1^{y/t-1}\ln s\, ds \le 2t + y\ln(1+y/t) \le 12 y\ln (1+y/t) .
\]
\end{proof}
\begin{proof}[Proof of Theorem \ref{T4.2}]
We will prove here that the one-sided right approximate derivative $D_{\textrm{ap}}^+\gamma$
exists almost everywhere and \eqref{Dapgamma} holds. The same argument can then
be used to establish the corresponding statement about the left derivative, and
these two statements together will give the full claim. Here, one-sided derivatives
are defined in the obvious way; for example, we say that $(D_{\textrm{ap}}^+f)(x)=d$ if
for all $\epsilon>0$,
\[
\lim_{h\to 0+}\frac{1}{h} \left| \left\{
y\in(0,h): \left| \frac{f(x+y)-f(x)}{y} - d \right| \ge \epsilon
\right\} \right| = 0 .
\]

Our basic strategy is modelled on the proof of \cite[Theorem 5.4]{MelPVol}.
The following statements hold at (Lebesgue) almost every
point $x\in\R$:
\begin{itemize}
\item $x$ is a Lebesgue point of $k'(t)$;
\item $\lim_{y\to 0+}g(x+iy)$ exists;
\item $\lim_{y\to 0+} k_s([x-y,x+y])/y=0$, where
$k_s$ is the singular part of $dk$;
\item
$\lim_{y\to 0+}\textrm{Im }g_s(x+iy)=0$, where $g_s(z)=\int_{\R}\frac{dk_s(t)}{t-z}$.
\end{itemize}
We will now show that if $x$ has all these properties,
then $(D_{\textrm{ap}}^+\gamma)(x)$ exists and \eqref{Dapgamma} holds.

So fix such an $x$. To simplify the notation, we will again assume that $x=0$. The basic
idea is to look at averages of
\[
F(y)\equiv \textrm{Re }g(iy) + \frac{\gamma(y)-\gamma(0)}{y} .
\]
By the formulae from Theorem \ref{T3.2}, $F(y)=\int_{\R} \phi_y(t)\, dk(t)$, where
\begin{align}
\label{4.6}
\phi(t) & = \frac{t}{t^2+1} + \ln \left| 1-\frac{1}{t} \right| ,\\
\nonumber
\phi_y(t) & = \frac{1}{y} \phi\left( \frac{t}{y} \right) .
\end{align}
Note that since $\phi(t)=O(t^{-2})$ for large $|t|$, we have that
$\phi,\phi_y\in L_1(\R)$. For later use, we also observe that
\begin{equation}
\label{4.5}
\int_{-\infty}^{\infty}\phi(t)\, dt = 0 .
\end{equation}
To prove this, look at $\int_{-R}^R \phi$. Clearly, the first term from
\eqref{4.6} is odd and thus doesn't contribute to this integral, and
\begin{align*}
\int_{-R}^R \ln \left| 1 - \frac{1}{t} \right| \, dt & =
\int_{-R}^R \ln \left|\frac{t-1}{t} \right| \, dt \\
& = \int_{-R-1}^{-R} \ln |t|\, dt - \int_{R-1}^R \ln |t|\, dt \to 0 \quad (R\to\infty) ,
\end{align*}
so we obtain \eqref{4.5}.

Suppose now that $B_y$ is a family of Borel sets with the following properties:
\begin{equation}
\label{By}
B_y \subset [\delta y, y], \quad \left| B_y \right| \ge \delta y ,
\end{equation}
for some fixed (but arbitrary) $0<\delta<1/2$. Define
\[
\psi_y(t) = \frac{1}{|B_y|} \int_{B_y} \phi_h(t)\, dh .
\]
We now claim that
\begin{equation}
\label{psiy}
\left| \psi_y(t) \right| \lesssim \begin{cases} y^{-1}\ln(1+y/|t|) & 0<|t|\le 2y \\
y/t^2 & |t|>2y \end{cases} .
\end{equation}
The constant implicit in \eqref{psiy} only depends on $\delta$. Indeed,
for $|t|\le 2y$ this follows immediately from Lemma \ref{L4.1} and the obvious bound
$|t|/(t^2+y^2)\le 2/y$ (if $|t|\le 2y$). If, on the other hand, $h\le y< |t|/2$, then
Taylor's theorem shows that
\[
\left| \phi_h(t) \right| = \frac{h^2}{|t|(t^2+h^2)} + O(h/t^2) \lesssim \frac{y}{t^2} ,
\]
and the second bound from \eqref{psiy} follows.

Next, \eqref{4.5} and the Fubini-Tonelli Theorem imply that
\begin{equation}
\label{4.10}
\int_{-\infty}^{\infty} \psi_y(t)\, dt = 0 .
\end{equation}
Our next goal is to show that
\[
\lim_{y\to 0+} \int_{-\infty}^{\infty} \psi_y(t)\, dk(t) = 0 .
\]
We rewrite this as
\begin{equation}
\label{4.8}
\left| \int_{\R} \psi_y(t)\, dk(t) \right| \le \int_{\R} \left| \psi_y(t)\right|\, dk_s(t)
+ \left| \int_{\R} \psi_y(t) k'(t)\, dt \right| .
\end{equation}
Our first step will be to show that
the first integral on the right-hand side of \eqref{4.8} goes to zero as $y\to 0$.
Start by considering the contributions coming from $|t|>2y$: By \eqref{psiy},
\[
\int_{|t|>2y} \left| \psi_y(t)\right|\, dk_s(t) \lesssim \int_{\R} \frac{y}{t^2+y^2}\, dk_s(t)
= \textrm{Im }g_s(iy) \to 0 ,
\]
by our choice of $x$ ($=0$). Next, if $\epsilon>0$ is given, we can find $\eta>0$ so that if
$h\le\eta$, then $k_s([-h,h])\le \epsilon h$. If $2y<\eta$, then this, \eqref{psiy}, and
the monotone convergence theorem imply that
\begin{align*}
\int_{|t|\le 2y} \left| \psi_y(t)\right|\, dk_s(t) & =
\sum_{n=0}^{\infty} \int_{2^{-n}y < |t| \le 2^{-n+1}y} \left| \psi_y(t)\right|\, dk_s(t) \\
& \lesssim \frac{1}{y} \sum_{n=0}^{\infty}
k_s\left( \left[ -2^{-n+1}y,2^{-n+1}y\right] \right) \ln \left( 1+2^n\right) \\
& \lesssim \epsilon \sum_{n=0}^{\infty}
(n+1)2^{-n} = C\epsilon .
\end{align*}
So $\limsup_{y\to 0+} \int_{\R} |\psi_y|\, dk_s \le C\epsilon$, but $\epsilon>0$
is arbitrary here, so the first integral from the right-hand side of \eqref{4.8}
goes to zero.

As for second integral, we recall \eqref{4.10} to estimate this as follows:
\[
\left| \int_{\R} \psi_y(t) k'(t)\, dt \right| \le \int_{\R} \left| \psi_y(t) \right|
| k'(t)-k'(0)|\, dt
\]
Now a very similar argument works, so we will just give a sketch of this.
First of all, for $|t|>2y$, $\psi_y(t)$ is dominated by
the Poisson kernel $y/(t^2+y^2)$, so this part goes to zero because $x=0$ is a Lebesgue
point of $k'$. For small $|t|$, on the other hand, we again have that the contributions
coming from $|t|\approx 2^{-n}y$ will be $\lesssim \epsilon n2^{-n}$, and the sum over
$n$ is still $\lesssim \epsilon$.

Let us summarize: We have shown that
\[
\lim_{y\to 0} \int \psi_y(t)\, dk(t) = 0 .
\]
By unwrapping the definitions,
we see that this means that
\[
\lim_{y\to 0} \left(
\frac{1}{|B_y|} \int_{B_y} \frac{\gamma(h)-\gamma(0)}{h} \, dh +
\frac{1}{|B_y|} \int_{B_y} \textrm{Re }g(ih) \, dh \right) = 0 .
\]
Since $g(ih)$ converges, to $g(0)$, by the choice of $x=0$ again, the second term
converges to $\textrm{Re }g(0)$, so we can also say that
\begin{equation}
\label{4.11}
\lim_{y\to 0} \frac{1}{|B_y|} \int_{B_y} \left( \frac{\gamma(h)-\gamma(0)}{h} + \textrm{Re }g(0)
\right) \, dh = 0 ,
\end{equation}
and this holds for any choice of sets $B_y$ as in \eqref{By}.
This implies that the (right) approximate derivative of $\gamma$ at $x=0$ exists
and \eqref{Dapgamma} holds.
Indeed, if this were \textit{not }true, then we could find $\delta, \epsilon>0$ and
a sequence of sets $A_n\subset[0, y_n]$, with $y_n\to 0$, such that $|A_n| \ge
3\delta y_n$ and
\[
\left| \frac{\gamma(h)-\gamma(0)}{h} + \textrm{Re }g(0) \right| \ge \epsilon
\]
for all $h\in A_n$.
But then we can also construct
sets $B_n\subset [\delta y_n, y_n]$, $|B_n|\ge \delta y_n$,
so that either
\[
\frac{\gamma(h)-\gamma(0)}{h} + \textrm{Re }g(0) \ge \epsilon
\]
for all $h\in B_n$ or $\ldots \le -\epsilon$ for all $h\in B_n$.
However, then \eqref{4.11} with $B_{y_n}=B_n$ leads to a contradiction, so we have
to admit that the (one-sided) approximate derivative exists and \eqref{Dapgamma} holds, as claimed.
\end{proof}
\begin{proof}[Proof of Theorem \ref{T4.3}]
This is similar to the previous proof. Again, we will explicitly discuss only
the point $x=0$, to simplify the notation.

We now define
\[
\phi(t) = \ln \left| 1- \frac{1}{t} \right| .
\]
Then, by the Thouless formula \eqref{thou},
\[
\gamma(y)-\gamma(0) = \int_{-\infty}^{\infty} \phi(t/y)\, dk(t) .
\]
We will again consider averages of this, but will have to set things up differently now.
Suppose the claim of the Theorem were wrong, that is,
\begin{equation}
\label{4.14}
\textrm{\rm cap}(B_n) \ge \delta y_n^N
\end{equation}
for some $\delta>0$, $N\ge 1$, and a sequence $y_n\to 0$ with $0<y_n\le 1/4$ (say) and sets
of the form
\[
B_n = \{ t\in (-y_n,y_n): \gamma(t)\le \gamma(0)-\epsilon \} .
\]
We are also using the fact that $\gamma$ is upper semicontinuous here,
that is, $\gamma(0)\ge \limsup_{x\to 0} \gamma (x)$, and thus for fixed
$\epsilon>0$ and small $t$, we can never have that $\gamma(t)\ge \gamma(0)+\epsilon$.

By \eqref{4.14}, the definition of the capacity of a set, and Frostman's Theorem
\cite[Theorem 3.3.4(a)]{Ran}, we can now find compact sets $K_n\subset B_n$ and
probability measures $\mu_n$ on $K_n$ (more precisely, these will be the equilibrium
measures of the sets $K_n$), such that
\[
\int_{\R} \ln|s-t|\, d\mu_n(s) \ge C\ln y_n
\]
for all $t\in\R$ and $n\in\N$. Now define
\[
\psi_n(t) = \int_{\R} \phi\left( \frac{t}{h} \right) \, d\mu_n(h) .
\]
Then, if $0<|t|\le 2y_n$ (recall also that $2y_n\le 1/2$),
\[
\left| \psi_n(t)\right| \le -\int_{\R} \ln |t-h|\, d\mu_n(h) - \ln |t| \lesssim -\ln |t| .
\]
Recall that $\ln |t|\in L_1(dk)$; indeed, this is part of what Corollary \ref{C3.1} asserts. Therefore,
by dominated convergence,
\[
\lim_{n\to\infty} \int_{|t|\le 2y_n} \left| \psi_n(t) \right| \, dk(t) = 0 .
\]
On the other hand, if $|t|>2y_n$ and $0<h\le y_n$, then $|\phi(t/h)|\le\ln 2$, and thus
also $|\psi_n(t)|\le\ln 2$ for these $t$. Moreover, $\phi(t/h)\to 0$ as $h\to 0+$ for
fixed $t\not= 0$, so $\lim_{n\to\infty}\psi_n(t)=0$. Therefore,
the dominated convergence theorem now shows that
\[
\lim_{n\to\infty} \int_{|t|> 2y_n} \left| \psi_n(t) \right| \, dk(t) = 0 ,
\]
too. So $\int_{\R} \psi_n\, dk \to 0$, but
\[
\int_{\R} \psi_n(t)\, dk(t) = \int_{\R} (\gamma(h)-\gamma(0))\, d\mu_n(h) \le -\epsilon
\]
by construction of $\mu_n$. It turns out that \eqref{4.14} is not tenable.
\end{proof}
\section{Some potential theory}
The Thouless formula says that $\gamma$ is essentially the logarithmic potential
of $dk$, so it's not surprising that notions from potential theory become relevant here.
In this section, we collect some results that can be obtained quite easily in this way.
While very little here is really new, we do feel that these facts complement and
illuminate the previous discussion. In particular, we will see that quite often,
the density of states will be the equilibrium measure of its support.

As mentioned above, the use of potential
theoretic notions in spectral theory is also the subject of a recent survey by Simon
\cite{Simpot}. This paper in fact motivated most of what we do in this section and
may be consulted for further information. Other useful references are \cite{SaT,StT}.

As before, we fix a sequence $N_j\to\infty$ for which the limits
from Theorem \ref{T3.1} exist. Since this sequence is otherwise
arbitrary, our results will again apply to \textit{all }possible limit points.

We introduce some notation: $K=\sigma(dk)$ will
denote the topological support of $dk$, that is, the smallest closed support of $dk$.
From its definition and Proposition \ref{P3.1},
we see that $K$ is compact and \textit{potentially perfect }in the
sense that $K_{\textrm{cap}}=K$.
If $E\subset\R$ is compact and of positive capacity,
we will denote its equilibrium measure by $\omega_E$.
This is defined as the (unique) probability (Borel) measure supported by $E$
that maximizes
\[
I(\nu) \equiv \int\!\!\int \ln |s-t|\, d\nu(s)\, d\nu(t)
\]
among all such measures. See \cite[Section 3.3]{Ran}. Alternatively,
$\omega_E$ may also be described as the harmonic measure for the region
$\C_{\infty}\setminus E$ and the point $\infty$. Also, by definition, for
a compact set $E\subset\R$, we have that
$\textrm{cap}(E)=e^{I(\omega_E)}$.
\begin{Theorem}
\label{T5.1}
The following conditions are equivalent:\\
(a) $\gamma(t)=\alpha$ for quasi every $t\in K=\sigma(dk)$;\\
(b) $\gamma(t)=\alpha$ for $\omega_K$-almost every $t$;\\
(c) $dk=d\omega_K$

In this case, $\alpha=\ln(\textrm{\rm cap}(K)/A)$.
\end{Theorem}
\begin{proof}
Obviously, (a) implies (b).
If (b) holds, then we may integrate the Thouless formula (see Corollary \ref{C3.1})
with respect to $d\omega_K$ and use the Fubini-Tonelli Theorem to obtain that
\[
\alpha = -\ln A + \int_K dk(t)\int_K d\omega_K(s)\, \ln |s-t| = \ln(\textrm{cap}(K)/A) .
\]
The last equality follows because $\int_K \ln |s-t|\, d\omega_K(t) = \ln \textrm{cap}(K)$
quasi everywhere on $K$ by Frostman's Theorem \cite[Theorem 3.3.4(b)]{Ran}, and
$k$ doesn't charge capacity zero sets by Proposition \ref{P3.1}.

On the other hand, by integrating with respect to $dk$, we find that
\[
I(dk) \equiv \int_{\R}dk(s)\int_{\R} dk(t)\, \ln |s-t| = \ln A + \int_{\R} \gamma(t)\, dk(t)
\ge \ln A + \alpha .
\]
The inequality follows because $\gamma\ge\alpha$ on $K$, and this can be seen as follows:
First of all, $\textrm{cap}(K\cap I)>0$ for all open intervals $I$ that intersect $K$,
and thus also $\omega_K(K\cap I)>0$ for all such $I$. Now $\gamma$ on $K$ is different
from $\alpha$ only on an $\omega_K$-null set, so $\limsup_{t\to x}\gamma(t)\ge\alpha$
for all $x\in K$ and now upper semicontinuity implies that $\gamma(x)\ge\alpha$ on $K$, as claimed.

So $I(dk)\ge \ln\textrm{\rm cap}(K)$, and thus we must have equality here, and (c) holds and
$\alpha$ has the asserted value.

Finally, if (c) holds, then, by Frostman's Theorem and the Thouless formula,
\[
\gamma(t) = -\ln A + \ln\textrm{\rm cap}(K)
\]
quasi everywhere on $K$.
\end{proof}
\begin{Theorem}
\label{T5.2}
Let $E=\sigma_{\textrm{\rm cap}}(J)$, and let $\rho$ be the spectral measure of $J$.
If $d\omega_E\ll d\rho$, then the limits from Theorem \ref{T3.1} exist as
$N\to\infty$ (without passing to a subsequence). In particular, they are unique, and in fact
$dk=d\omega_E$, and $\gamma=0$ quasi everywhere on $E$.
\end{Theorem}
\begin{proof}
Since $\omega_E\ll \rho$ and $\omega_E$ doesn't give weight to capacity zero sets,
Theorem \ref{T4.1} implies that $\omega_E$ is supported by $Z=\{ \gamma = 0\}$. Thus
integration of the Thouless formula with respect to $d\omega_E$ gives
\[
0=-\ln A + \int_{\R} dk(t) \int_{\R} d\omega_E(s)\, \ln |s-t| = - \ln A + \ln\textrm{\rm cap}(E),
\]
by Frostman's Theorem and Proposition \ref{P3.1} again. On the other hand, integration
with respect to $dk$ shows that
\[
I(dk) = \ln A + \int_{\R} \gamma(t)\, dk(t) \ge \ln A = \ln\textrm{\rm cap}(E) ,
\]
so $dk=d\omega_E$. Theorem \ref{T5.1} now shows that $\gamma=0$ quasi everywhere on $E$.
\end{proof}
Theorem \ref{T5.2} in particular says that equilibrium measures arise
as the unique density of states measures in many situations, and
$\omega_E\in\mathcal R_0(E)$ by Corollary \ref{C4.1}.

Next, we briefly touch the subject of Denisov-Rakhmanov type theorems. For a compact,
essentially closed set $E\subset\R$, define $\mathcal {DR}(E)$ as the set of
bounded half line
Jacobi matrices $J$ that satisfy
\[
\sigma_{ess}(J)=\Sigma_{ac}(J)=E .
\]
A set $E$ is called essentially closed if it is equal to its \textit{essential closure}
\[
\overline{E}^{ess} = \left\{ x\in\R : |E\cap (x-h,x+h)| > 0
\textrm{ for all }h>0 \right\} .
\]
This terminology is common but somewhat unfortunate because $\overline{E}^{ess}$
really is the set of accumulation points of $E$ with respect to the topology
with basis $(a,b)\setminus N$, $|N|=0$. Also, note the formal analogy to \eqref{Scap}.

At present, Denisov-Rakhmanov Theorems are available only
for rather special sets $E$ (for finite unions of closed intervals and
homogeneous compact sets). These results give rather detailed
information on the asymptotic behavior of the coefficients of an arbitrary $J\in
\mathcal {DR}(E)$. See \cite{DKS,Den,Remac}. The case of more general sets $E$ is
quite unclear and actually the subject of current research. As pointed out in \cite{Simpot},
potential theoretic tools are ideal to produce poor man's versions of Denisov-Rakhmanov
Theorems, which have much less detailed conclusions, but, on the plus side, work
very generally. Here is such a result; see also the very similar discussion of Widom's Theorem
\cite{Wid} in \cite[Section 4]{Simpot}.
\begin{Theorem}
\label{T5.3}
Suppose that $E$ is an essentially closed, bounded subset of $\R$, and
suppose that $\left( \omega_E\right)_s=0$. If $J\in\mathcal {DR}(E)$, then
the limits from Theorem \ref{T3.2} exist as $N\to\infty$, and
$dk=d\omega_E$ and $\gamma=0$
quasi everywhere on $E$.

In particular,
this will hold if $E$ is a
compact, weakly homogeneous subset of $\R$ and $J\in\mathcal {DR}(E)$.
\end{Theorem}
As the following proof will show, we may actually replace the assumption
that $J\in\mathcal {DR}(E)$ by the slightly weaker condition
\[
\sigma_{\textrm{cap}}(J)=\Sigma_{ac}(J)=E .
\]
\begin{proof}
$E=E_{\textrm{cap}}$ since $E$ is essentially closed, and we are assuming that
$d\rho_{ac}(t)\sim \chi_E(t)\, dt$, so we are clearly in the situation of Theorem \ref{T5.2}.
The last part follows from Corollay \ref{C2.1} because $\omega_E$ is reflectionless on $E$.
\end{proof}
Finally, we collect some general inclusions and inequalities; most of these were
already obtained above, but it seems useful to have them readily available. Recall
that we defined $K$ as the topological support of $dk$ and
$Z=\{t\in\R :\gamma(t)=0 \}$, $A=\lim_{j\to\infty}(a(1)\cdots a(N_j))^{1/N_j}$.
The absolutely continuous spectrum, $\sigma_{ac}$, can be obtained from
$\Sigma_{ac}$ as $\sigma_{ac}=\overline{\Sigma_{ac}}^{ess}$.
\begin{Theorem}
\label{T5.4}
We have that
$Z\supset\Sigma_{ac}(J)$ and $\overline{Z}^{\textrm{ess}}\supset \sigma_{ac}(J)$, and
\[
\overline{Z}^{\textrm{ess}} \subset K \subset \sigma_{\textrm{\rm cap}}(J) , \quad\quad
\textrm{\rm cap}(Z) \le A \le \textrm{\rm cap}(K) .
\]
\end{Theorem}
\begin{proof}
We already observed in the proof of Theorem \ref{T4.1} that $\Sigma_{ac}\subset Z$,
so $\sigma_{ac}=\overline{\Sigma_{ac}}^{ess}\subset \overline{Z}^{ess}$. Proposition \ref{P3.1}
informs us that $K\subset\sigma_{\textrm{cap}}$. By Corollary \ref{C4.1},
$g\in\mathcal N(Z)$, so $\chi_Z\, dt\ll dk$ and hence $\overline{Z}^{ess}\subset K$.
If $B\subset Z$ is compact, $\textrm{cap}(B)>0$,
then integration of the Thouless formula yields
\[
0 = -\ln A + \int_{\R}dk(t) \int_{\R} d\omega_B(s) \ln |s-t| \ge -\ln A + \ln\textrm{\rm cap}(B) ,
\]
so $\textrm{cap}(Z)\le A$. On the other hand,
\[
I(dk) = \ln A + \int_{\R} \gamma(t)\, dk(t) \ge \ln A ,
\]
and $dk$ is supported by $K$, so $\textrm{cap}(K)\ge A$.
\end{proof}
Since the limits from Theorem \ref{T3.1} can be taken on sub-subsequences
of arbitrary subsequences, this in particular says the following:
\begin{Corollary}
\label{C5.1}
Let
\begin{align*}
A_- & = \liminf_{N\to\infty} \left( a(1)a(2) \cdots a(N) \right)^{1/N} , \\
A_+ & = \limsup_{N\to\infty} \left( a(1)a(2) \cdots a(N) \right)^{1/N} .
\end{align*}
Then $\textrm{\rm cap}(\Sigma_{ac})\le A_-$ and $\textrm{\rm cap}(\sigma(J))
\ge A_+$; also, $|\Sigma_{ac}|\le 4A_-$.
\end{Corollary}
The bound on $\textrm{cap}(\Sigma_{ac})$ must be interpreted carefully because $\Sigma_{ac}$
is only defined up to sets of Lebesgue measure zero. From the proof of Theorem \ref{T5.4},
however, it is clear how to proceed: we have that $\textrm{cap}(S)\le A_-$ for \textit{some}
set $S$ with $|S\triangle\Sigma_{ac}|=0$, and in fact it suffices to take $S\subset Z$ to make
sure that this inequality holds.

The last part Corollary \ref{C5.1} follows because $|S|\le 4\, \textrm{cap}(S)$ for Borel sets
$S\subset\R$ \cite[Theorem 5.3.2(c)]{Ran}. In the ergodic setting, this result
was obtained (much) earlier in \cite{DeiftSim}.

\end{document}